\documentclass{elsarticle}
\usepackage{amsfonts}
\usepackage{amsmath}
\usepackage{amsthm}
\usepackage[all]{xy}
\usepackage{tikz}
\usetikzlibrary{trees}
\usetikzlibrary{arrows,shapes,snakes,automata,backgrounds,petri}
\usepackage{bm}
\usepackage{cleveref}
\usepackage{url}
\DeclareMathOperator{\tr}{tr}

\newtheorem{thm}{Theorem}
\newtheorem{cor}{Corollary}

\theoremstyle{remark}
\newtheorem{rem}{Remark}

\begin{document}
\begin{frontmatter}
\title{Adding species to chemical reaction networks: preserving rank preserves nondegenerate behaviours}

\author[1]{Murad Banaji}
\author[2]{Bal\'azs Boros\corref{cor1}}
\author[2]{Josef Hofbauer}

\address[1]{Department of Design Engineering and Mathematics, Middlesex University London}
\address[2]{Department of Mathematics, University of Vienna}

\cortext[cor1]{BB's work was supported by the Austrian Science Fund (FWF), project P32532.}

\begin{abstract}
We show that adding new chemical species into the reactions of a chemical reaction network (CRN) in such a way that the rank of the network remains unchanged preserves its capacity for multiple nondegenerate equilibria and/or periodic orbits. One consequence is that any bounded nondegenerate behaviours which can occur in a CRN can occur in a CRN with bounded stoichiometric classes. The main result adds to a family of theorems which tell us which enlargements of a CRN preserve its capacity for nontrivial dynamical behaviours. It generalises some earlier claims, and complements similar claims involving the addition of reactions into CRNs. The result gives us information on how ignoring some chemical species, as is common in biochemical modelling, might affect the allowed dynamics in differential equation models of CRNs. We demonstrate the scope and limitations of the main theorem via several examples. These illustrate how we can use the main theorem to predict multistationarity and oscillation in CRNs enlarged with additional species; but also how the enlargements can introduce new behaviours such as additional periodic orbits and new bifurcations.
\end{abstract}
\begin{keyword}
Oscillation; multistationarity; chemical reaction networks; bifurcations

\smallskip
\textbf{MSC.} 92E20; 37C25; 34D10
\end{keyword}

\end{frontmatter}

\section{Introduction and outline of the main result}

An important theme in the mathematical study of chemical reaction networks (CRNs) relates to how network structure influences network dynamics. The results in this direction sometimes allow us to infer detailed information on dynamical behaviours of reaction networks using only graph theory and linear algebra. We may, for example, be able to conclude from basic computations that a CRN has very simple behaviour, such as the convergence of all initial conditions to an equilibrium, regardless of parameter values. In the opposite direction, we may be able to state, without numerical simulation, that a CRN admits some interesting behaviour such as stable oscillation, and even know {\em a priori} how to choose parameter values to obtain this behaviour. 

Amongst more complicated behaviours which occur in ordinary differential equation models of CRNs, the most well-studied are multistationarity and oscillation. The history of study of these behaviours in the context of biological modelling is reviewed in \cite{switchclock}. Crucially, multistationarity and oscillation are not just of abstract interest, but may be of functional importance in biological switching and signalling processes \cite{Kholodenko.2000aa,markevich,Ortega.2006ab,Qiao.2007aa,novaktyson,cheong, obatake}. For this reason, results which tell us which structures in a CRN guarantee such behaviours are of considerable interest. 

A family of theorems termed ``inheritance results'' tell us when a CRN is guaranteed to exhibit some dynamical behaviour simply because of the presence of a certain subnetwork. Examples of such results can be found in \cite{joshishiu, feliuwiufInterface2013, banajipanteaMPNE, banajiCRNosci, banajiCRNosci1}. Such conditions which {\em guarantee} nontrivial behaviours based on CRN structure are dual to claims which {\em rule out} nontrivial behaviours in a CRN. There is a large classical and modern literature on conditions which preclude certain dynamical behaviours in CRNs. Examples include \cite{hornjackson, feinberg, angelileenheersontag, craciun, shinarfeinbergconcord1, banajicraciun2, feliuwiufAMC2012, abphopf}. Each new result in either direction helps to narrow the gap between conditions guaranteeing, and conditions ruling out, nontrivial behaviours in CRNs.

The main result of this paper is a new inheritance result, which is both natural and relatively straightforward to prove. The theorem is most simply stated in terms of the {\em rank} of a CRN. Each reaction of a CRN on $n$ chemical species defines a reaction vector, a real (often integer) $n$-vector whose $k$th entry tells us the net production of species $k$ in the reaction. The span of these vectors is the {\em stoichiometric subspace} of the CRN, whose dimension is defined to be the rank of the CRN. In any model of the CRN, numbers or concentrations of the chemical species are confined to affine subspaces parallel to the stoichiometric subspace. The nonnegative portions of these affine subspaces are termed the {\em stoichiometric classes} of the system. 

The rank of a CRN figures in various aspects of the theory. For example, it plays a crucial role in the original results of deficiency theory \cite{feinberg}, and indeed in the definition of the deficiency of a reaction network. More recently, various results have been proved for CRNs of sufficiently low rank, regardless of how many species or reactions are involved. Characterisations of rank-1 CRNs admitting multistationarity are given in \cite{lin2021multistationarity}. In rank-2 CRNs, the Poincar\'e-Bendixson theorem \cite[Chapter 9]{Wiggins} can be used to rule out chaos or guarantee oscillation; and the same can sometimes be extended to rank-3 CRNs, a fact exploited in the analysis in \cite{HalSmithJMC} and \cite{boros:hofbauer:2022b}. In \cite{panteapersistence}, the famous ``global attractor conjecture'' is proved for rank-3 CRNs. 

Stated informally, we have the following complementary inheritance results involving the rank of a CRN. Both are simple applications of regular perturbation theory:
\begin{enumerate}
\item Adding new {\bf reactions} into a CRN without changing its rank preserves its capacity for nontrivial behaviours including nondegenerate multistationarity and oscillation. This was proved in previous work (Theorem~1 in \cite{banajipanteaMPNE} and Theorem~1 in \cite{banajiCRNosci}).
\item Adding new {\bf species} into a CRN without changing its rank preserves its capacity for nontrivial behaviours including nondegenerate multistationarity and oscillation. This is the content of Theorem~\ref{mainthm} here. 
\end{enumerate}

These claims demonstrate that in CRN theory we often have complementary results where we can interchange ``species'' and ``reactions''. Note, however, that adding species into a CRN while preserving its rank can result in some fairly fundamental changes to the CRN. For example, stoichiometric classes of the enlarged CRN may be bounded even if those of the original were unbounded. This is in contrast to adding linearly dependent reactions, a process which leaves stoichiometric classes unchanged. 

Taken together, the above claims imply that building a CRN without altering its rank preserves its capacity for nondegenerate dynamical behaviours. This claim is stated more formally as Corollary~\ref{corrank} later. Note that the enlarged CRNs may, of course, admit more complicated and interesting behaviours than the original, as we shall see by example. 

After some preliminary definitions, we will present the statement and proof of the main theorem and several remarks on its implications and generalisations. This is followed by several examples which demonstrate both the main result and its limitations. 

\section{Preliminaries}

We present some key notions briefly. A much more expansive treatment of the main background can be found in previous work \cite{banajipantea,banajipanteaMPNE,banajiCRNosci}. We consider a CRN to be an ordered set of chemical species and an ordered set of reactions, where the orderings are arbitrary but fixed. Each reaction is an ordered pair of {\em complexes}, namely formal linear combinations of chemical species. The coefficient of each species in a complex is taken to be nonnegative, and often to be an integer, although the latter is not required here. The pair of complexes which define a reaction are termed the ``reactant complex'' and ``product complex'' of the reaction.

{\bf Positive sets in Euclidean space.} The {\em positive orthant} in $\mathbb{R}^n$ is denoted by $\mathbb{R}^n_{+}$ and defined as $\{x \in \mathbb{R}^n \,:\, x_i > 0\,\, \mbox{for}\,\, i = 1, \ldots, n\}$. A set in $\mathbb{R}^n$ is termed {\em positive} if it lies in $\mathbb{R}^n_{+}$. The closure of $\mathbb{R}^n_{+}$ is denoted by $\mathbb{R}^n_{\geq 0}$ and referred to as the {\em nonnegative orthant} in $\mathbb{R}^n$, namely $\mathbb{R}^n_{\geq 0} = \{x \in \mathbb{R}^n \,:\, x_i \geq 0\,\, \mbox{for}\,\, i = 1, \ldots, n\}$.

{\bf Stoichiometric matrix and stoichiometric classes.} Each reaction is associated with a {\em reaction vector} whose $k$th entry is the net production of the $k$th species in the reaction: this is just the stoichiometric coefficient of the $k$th species in the product complex minus its stoichiometric coefficient in the reactant complex. The {\em stoichiometric matrix} of a CRN is the matrix whose $j$th column is the $j$th reaction vector of the CRN. Given a CRN $\mathcal{R}$ on $n$ species with stoichiometric matrix $\Gamma$, the span of the columns of $\Gamma$ (a linear subspace of $\mathbb{R}^n$) is denoted by $\mathrm{im}\,\Gamma$, and is termed the {\em stoichiometric subspace} of $\mathcal{R}$. The rank of $\Gamma$ is termed the {\em rank} of $\mathcal{R}$. The nonnegative parts of cosets of $\mathrm{im}\,\Gamma$, namely sets of the form $(x+\mathrm{im}\,\Gamma) \cap \mathbb{R}^n_{\geq 0}$ ($x \in \mathbb{R}^n_{\geq 0}$), are the {\em stoichiometric classes} of $\mathcal{R}$. The positive parts of stoichiometric classes, namely sets of the form $(x+\mathrm{im}\,\Gamma) \cap \mathbb{R}^n_{+}$ ($x \in \mathbb{R}^n_{+}$), are the {\em positive stoichiometric classes} of $\mathcal{R}$. 

{\bf Ordinary differential equations (ODEs) and rate functions.} Any system of ODEs describing the evolution of a CRN with stoichiometric matrix $\Gamma$ takes the form $\dot x = \Gamma v(x)$. The function $v$ is termed the {\em rate function} of the CRN and its $j$th component tells us how the rate of the $j$th reaction depends on the concentrations of the chemical species. In this paper, $v$ is assumed, as a minimum, to be defined and continously differentiable on $\mathbb{R}^n_{+}$. 

{\bf Nondegenerate and linearly stable limit sets.} Consider a CRN of rank $r$ with stoichiometric matrix $\Gamma$. Let $\mathcal{S}$ be some coset of $\mathrm{im}\,\Gamma$ containing a positive equilibrium (resp., periodic orbit) $\mathcal{O}$. Since the positive part of $\mathcal{S}$ is locally invariant, $\mathcal{O}$ has exactly $r$ eigenvalues (resp., Floquet multipliers) relative to $\mathcal{S}$. If none of these are equal to zero (resp., exactly one of these is equal to one), then we say that $\mathcal{O}$ is {\em nondegenerate}. (Note that this terminology differs from that in \cite{banajiCRNosci}, where an invariant set was referred to as ``nondegenerate'' only if it was hyperbolic relative to its stoichiometric class.) In an abuse of terminology, we refer to $\mathcal{O}$ as {\em hyperbolic} if it is hyperbolic relative to its stoichiometric class, and {\em linearly stable} if it is linearly stable relative to its stoichiometric class.

{\bf Continuation of nondegenerate limit sets.} Suppose $U$ is some open region in $\mathbb{R}^n$, $a > 0$, and $f: U \times (-a,a) \to \mathbb{R}^n$ is $C^1$. If $\dot x = f(x, 0)$ has a nondegenerate equilibrium (resp., periodic orbit) on $U$, then the same is true for $\dot x = f(x, \varepsilon)$ for all $\varepsilon$ sufficiently small. The same conclusion holds if we replace ``nondegenerate'' by ``hyperbolic'' or ``linearly stable''. The results follow from the implicit function theorem and the fact that eigenvalues of a matrix depend continuously on its entries. The claim for equilibria is an immediate consequence of the implicit function theorem (see, for example, Corollary~6.9 in \cite{banajipanteaMPNE}), while the details for periodic orbits are laid out in Section~IV in \cite{Fenichel79}. An immediate consequence is that nondegenerate equilibria and periodic orbits of $\dot x = \Gamma v(x)$ on some positive stoichiometric class survive sufficiently small $C^1$ perturbations to $v(x)$ (here $v$ is assumed to be at least $C^1$). In fact, if we restrict attention to hyperbolic equilibria and periodic orbits, then the results for equilibria and periodic orbits are special cases of more general results on the persistence of normally hyperbolic invariant manifolds in \cite{Fenichel71} and \cite{HPS77}. Thus, the main result here generalises naturally to the case of invariant manifolds which are normally hyperbolic relative to their stoichiometric class.

{\bf Entrywise products and generalised monomials.} The notation $a \circ b$ denotes the {\em entrywise product} of matrices or vectors $a$ and $b$ assumed to have the same dimensions. Given $x=(x_1,\ldots, x_n)^{\mathrm{t}}$ and $a = (a_1,\ldots, a_n)$, $x^a$ is an abbreviation for the {\em generalised monomial} $x_1^{a_1}x_2^{a_2}\cdots x_n^{a_n}$. Let $A_1, \ldots, A_m$ be the rows of an $m \times n$ matrix $A$. Each $x^{A_i}$ is then a generalised monomial, and $x^A$ denotes the vector of these monomials, namely, $(x^{A_1}, x^{A_2}, \ldots, x^{A_m})^{\mathrm{t}}$.

{\bf Kinetics and admitted behaviours.} When we restrict the rate function of a CRN $\mathcal{R}$ to some class of functions $\mathcal{K}$, the pair $(\mathcal{R}, \mathcal{K})$ is referred to as a {\em CRN with kinetics}. We can think of $(\mathcal{R}, \mathcal{K})$ as a set of allowed ODE models of the CRN. We say that $(\mathcal{R}, \mathcal{K})$ {\em admits} some particular dynamical behaviour if this behaviour occurs in some allowed model, i.e., for some choice of rate function from $\mathcal{K}$ and on some stoichiometric class. Otherwise the CRN with kinetics {\em forbids} this behaviour. Different classes of kinetics for CRNs are discussed in detail in \cite{banajiCRNosci}. 

{\bf Power-law kinetics.} Let $\mathsf{X}_1, \ldots, \mathsf{X}_n$ denote the chemical species of a CRN and $x_1, \ldots, x_n$ denote the concentrations of these species. If the $i$th reaction has power-law kinetics, this means that the $i$th rate function takes the form $v_i(x) = \kappa_i x^a$, where $\kappa_i$ is a positive constant (termed the {\em rate constant} of the $i$th reaction), and $a$ is a real (row) vector, termed the vector of exponents for the reaction. If all reactions of a CRN have power-law kinetics, we can stack these row vectors into a matrix $A$, termed the {\em matrix of exponents} of the reaction network, whose $ij$th entry tells us the exponent of species $j$ in the rate function for reaction $i$. In this case, the rate function can be written briefly $\kappa \circ x^A$. If $A$ is fixed in advance, we say that the CRN has {\em fixed power-law kinetics}. 

{\bf Mass action kinetics.} Mass action kinetics is a special case of fixed power-law kinetics where $a_{ij}$, the $ij$th entry in the matrix of exponents, is precisely the stoichiometric coefficient of species $\mathsf{X}_j$ in the reactant complex of reaction $i$. 

{\bf Enlarging CRNs and inheritance.} We are often interested in claiming that whenever a CRN with kinetics admits some behaviour, then so does an enlarged CRN with kinetics in some related class. {\em Inheritance} results relate dynamical behaviours admitted in the enlarged CRNs to those admitted in the original CRNs. 

{\bf Adding linearly dependent species.} In this paper we are interested in an enlargement where $\mathcal{R}'$ is obtained from $\mathcal{R}$ by adding a {\em linearly dependent species} into the reactions of $\mathcal{R}$. In such a modification, the stoichiometric matrix of the CRN is unchanged except for the addition of a new row, and this new row is a linear combination of the existing rows of the stoichiometric matrix. We refer to the process of enlarging a CRN $\mathcal{R}$ by adding in some new, linearly dependent, species as {\em lifting}. The terminology is motivated by the fact that the addition of a new species increases the dimension of the state space by $1$, with the original, lower dimensional, state space naturally embedded in the new state space.

{\bf Derived power-law kinetics.} Suppose that $(\mathcal{R}, \mathcal{K})$ and $(\mathcal{R}', \mathcal{K}')$ are two CRNs, with $\mathcal{R}'$ obtained by adding new species and/or reactions to $\mathcal{R}$, and $\mathcal{K}$, $\mathcal{K}'$ being fixed power-law kinetics with matrices of exponents $A$ and $A'$ respectively. We then say that $\mathcal{K}'$ is {\em derived} from $\mathcal{K}$, if the submatrix of $A'$ corresponding to the original species and reactions of $\mathcal{R}$ is precisely $A$.

{\bf Permanence.} Consider a system of ODEs on some subset of $\mathbb{R}^n$ and suppose that $\mathcal{X}\subseteq \mathbb{R}^n_{\geq 0}$ is forward invariant for the system. The system is {\em permanent} on $\mathcal{X}$ if there exists a forward invariant, compact, positive set $\mathcal{Z} \subseteq \mathcal{X}$ such that the forward trajectory of every positive initial condition in $\mathcal{X}$ eventually enters $\mathcal{Z}$. In the context of CRNs we may think of $\mathcal{X}$ as a stoichiometric class: we may be interested in permanence on some or all stoichiometric classes.

\section{The main result}

Given a parameterised family of compact sets $\mathcal{X}_\varepsilon$ in Euclidean space, with $\varepsilon \in (0, a)$ for some $a > 0$, ``$\mathcal{X}_\varepsilon$ is close to $\mathcal{X}$'' will mean that given any $\delta>0$, there exists $\varepsilon_1 \in (0, a]$ such that for all $\varepsilon \in (0, \varepsilon_1)$ the Hausdorff distance between $\mathcal{X}_\varepsilon$ and $\mathcal{X}$ is less than $\delta$.

\begin{thm}
\label{mainthm}
Let $(\mathcal{R}, \mathcal{K})$ be a CRN with fixed power-law kinetics. Let $\mathcal{R}'$ be derived by adding to $\mathcal{R}$ a new linearly dependent species, and let $\mathcal{K}'$ be any fixed power-law kinetics for $\mathcal{R}'$ derived from $\mathcal{K}$. Suppose that for some choice of kinetics from $\mathcal{K}$, $\mathcal{R}$ has, on some stoichiometric class, at least $0 \leq r_1 < \infty$ positive, nondegenerate (resp., hyperbolic, resp., linearly stable) equilibria and at least $0 \leq r_2 < \infty$ positive, nondegenerate (resp., hyperbolic, resp., linearly stable) periodic orbits. Then, for some choice of kinetics from $\mathcal{K}'$, $\mathcal{R}'$ has, on some stoichiometric class, at least $r_1$ positive, nondegenerate (resp., hyperbolic, resp., linearly stable) equilibria and at least $r_2$ positive, nondegenerate (resp., hyperbolic, resp., linearly stable) periodic orbits.
\end{thm}
\begin{proof}
Let $\mathcal{R}$ include $n$ chemical species and $m$ irreversible reactions with stoichiometric matrix $\Gamma \in \mathbb{R}^{n\times m}$. By hypothesis, there exists $v \in \mathcal{K}$, such that the associated ODE system
\begin{equation}
\label{eq1}
\dot x = \Gamma v(x)\,
\end{equation}
has $r_1$ nondegenerate equilibria and $r_2$ nondegenerate periodic orbits on some positive stoichiometric class $\mathcal{S}_+$. The rate function $v$ is analytic on $\mathbb{R}^n_{+}$, and hence certainly $C^1$ on $\mathbb{R}^n_{+}$ which is, in fact, the only assumption we will need about $v$. Let $\mathcal{S}$ be the affine hull of $\mathcal{S}_+$, i.e., the coset of $\mathrm{im}\,\Gamma$ containing $\mathcal{S}_+$. 

Let $\mathcal{O}$ refer to a nondegenerate positive equilibrium (resp., periodic orbit) on $\mathcal{S}$. Choose $\mathcal{Z} \subseteq \mathcal{S}$ to be compact and positive, with $\mathcal{O} \subseteq \mathcal{Z}^o$, the relative interior of $\mathcal{Z}$ in $\mathcal{S}$. We may assume that similar sets are constructed around each of the $r_1+r_2$ nondegenerate equilibria and periodic orbits of (\ref{eq1}), and that these sets are pairwise disjoint. 

The hypothesis that the added species is linearly dependent implies that there exists $c \in \mathbb{R}^n$ such that the stoichiometric matrix of $\mathcal{R}'$ takes the form
\[
\Gamma' := \left(\begin{array}{c}\Gamma\\c^t\Gamma\end{array}\right)\,.
\]
Define $\varepsilon$ to be a positive parameter to be controlled, and let $\alpha_j \in \mathbb{R}\,\,(j=1,\ldots, m)$ be any real numbers. Denoting the concentration of the new species by $y$, set the rate of the $j$th reaction in $\mathcal{R}'$ to be
\[
v'_j(x,y, \varepsilon) = \varepsilon^{\alpha_j}\,v_j(x)\,y^{\alpha_j}\,.
\]
This choice of reaction rate corresponds to giving the new species exponent $\alpha_j$ in reaction $j$, and multiplying the original rate constant of the $j$th reaction by $\varepsilon^{\alpha_j}$. In brief notation, $\mathcal{R}'$ now has rate function $ v'(x,y,\varepsilon) := \varepsilon^{\alpha} \circ v(x) \circ y^{\alpha}$ where $\alpha := (\alpha_1, \ldots, \alpha_m)^t$. The evolution of $\mathcal{R}'$ is governed by
\begin{equation}
\label{eq2}
\left(\begin{array}{c}\dot x\\\dot y\end{array}\right) = \Gamma'\,v'(x,y,\varepsilon)\,.
\end{equation}

Note that $\mathcal{R}'$ has a new conservation law of the form $-c^tx+y=\mathrm{constant}$. For any fixed $\varepsilon>0$, we will focus our attention on the invariant set $\mathcal{H}_\varepsilon \subseteq \mathbb{R}^{n+1}$ defined by setting this constant to be equal to $\frac{1}{\varepsilon}$, namely, 
\[
\mathcal{H}_\varepsilon = \{(x, y) \in \mathbb{R}^n \times \mathbb{R}\,\colon\, y = \frac{1}{\varepsilon} + c^tx\}\,.
\]
The map
\[
h_\varepsilon\colon x \mapsto \left(\begin{array}{c}x\\\frac{1}{\varepsilon} + c^tx\end{array}\right)
\]
is an affine bijection between $\mathbb{R}^n$ and $\mathcal{H}_\varepsilon$. If $c=0$, set $\varepsilon_1 = 1$, and otherwise set
\[
\varepsilon_1 = \frac{1}{\sup_{x\in \mathcal{Z}}\,|c^tx|}\,.
\]
Then, for $\varepsilon \in (0, \varepsilon_1)$, $h_\varepsilon(\mathcal{Z})$ is a compact, positive subset of $\mathcal{H}_\varepsilon$. The map $h_\varepsilon$ defines local coordinates on $\mathcal{H}_\varepsilon$ which evolve according to
\begin{equation}
\label{eq3}
\dot x = \Gamma\, v'(x, \frac{1}{\varepsilon} + c^tx, \varepsilon) = \Gamma\,(\varepsilon^{\alpha} \circ v(x) \circ (\frac{1}{\varepsilon} + c^t x)^{\alpha}) = \Gamma\,(v(x) \circ (1 + \varepsilon\, c^t x)^{\alpha}) \,.
\end{equation}
Here $x$ refers to the local coordinate on $\mathcal{H}_\varepsilon$, rather than the original coordinate on $\mathbb{R}^n$: this should cause no confusion as we are identifying $\mathcal{H}_\varepsilon$ with $\mathbb{R}^n$ via $h_\varepsilon$. Note that the right hand side of (\ref{eq3}) is well-defined (and $C^1$) provided $x$ is positive, and $\varepsilon < \frac{1}{|c^tx|}$, which certainly holds on $\mathcal{Z}^o \times (-\varepsilon_1, \varepsilon_1)$. 

We wish to restrict our attention to $\mathcal{Z}^o$. We can, if desired, pass to local coordinates on $\mathcal{S}$ in a standard way (see the proofs of several results in \cite{banajiCRNosci}), but here this is unnecessary: we simply bear in mind that we are considering the restriction of (\ref{eq3}) to $\mathcal{Z}^o$, with $\varepsilon \in (-\varepsilon_1, \varepsilon_1)$.

Since the vector field in (\ref{eq3}) is a $C^1$ perturbation of that of (\ref{eq1}), by regular perturbation theory, there exists $\varepsilon_2 \in (0, \varepsilon_1]$ such that for each $\varepsilon \in (0, \varepsilon_2)$, (\ref{eq3}) has an equilibrium (resp., periodic orbit) $\mathcal{O}_\varepsilon$ in $\mathcal{Z}^o$, which is nondegenerate, and close to $\mathcal{O}$. (The details in the harder case where $\mathcal{O}$ is a periodic orbit are in Section~IV in \cite{Fenichel79}, for example.) If $\mathcal{O}$ is hyperbolic relative to $\mathcal{S}$, then we can choose $\varepsilon_2$ to ensure that the linear stability type of $\mathcal{O}_\varepsilon$ relative to $\mathcal{S}$ is the same as that of $\mathcal{O}$. More precisely, (i) if $\mathcal{O}$ is an equilibrium with, relative to $\mathcal{S}$, $k_1$ eigenvalues with positive real part, $k_2$ eigenvalues with negative real part, and no eigenvalues on the imaginary axis, then the same holds for $\mathcal{O}_\varepsilon$; (ii) if $\mathcal{O}$ is a periodic orbit with, relative to $\mathcal{S}$, $k_1$ Floquet multipliers inside the unit circle, $k_2$ outside the unit circle, and precisely one multiplier on the unit circle, then the same holds for $\mathcal{O}_\varepsilon$. As a special case, if $\mathcal{O}$ was linearly stable relative to $\mathcal{S}$, then the same holds for $\mathcal{O}_\varepsilon$.

Since $\mathcal{O}_\varepsilon$ lies in $\mathcal{Z}^o$, it is positive. Let $\mathcal{S}'_\varepsilon = h_\varepsilon (\mathcal{S})$. Explicitly, 
\[
\mathcal{S}'_\varepsilon := \left(\begin{array}{c}x_0\\\frac{1}{\varepsilon}+c^tx_0\end{array}\right) + \mathrm{im}\,\Gamma'\,,
\]
where $x_0$ is any element of $\mathcal{S}$. Note that $\mathcal{S}'_\varepsilon$ has the same dimension as $\mathcal{S}$. Clearly, $\mathcal{O}'_\varepsilon := h_\varepsilon(\mathcal{O}_\varepsilon)$ is an equilibrium (resp., periodic orbit) of (\ref{eq2}), and we have ensured (via the choice of $\varepsilon_1$) that $h_\varepsilon(\mathcal{Z}^o)$, and hence $\mathcal{O}'_\varepsilon \subseteq h_\varepsilon(\mathcal{Z}^o)$, are positive. As $h_\varepsilon$ is an affine bijection between $\mathcal{S}$ and $\mathcal{S}'_\varepsilon$, the choice of $\varepsilon_2$ ensures that the linear stability type of $\mathcal{O}'_\varepsilon$ relative to $\mathcal{S}'_\varepsilon$ is the same as that of $\mathcal{O}_\varepsilon$ relative to $\mathcal{S}$.

We can repeat the same argument in a neighbourhood of each of a finite number of nondegenerate equilibria or periodic orbits of (\ref{eq1}) on $\mathcal{S}_+$. By choosing $\varepsilon_2^*$ to be the minimum of the values of $\varepsilon_2$ associated with each limit set, we can ensure that provided $\varepsilon \in (0, \varepsilon_2^*)$, $\mathcal{R}'$ has at least $r_1$ positive, nondegenerate equilibria and at least $r_2$ positive, nondegenerate periodic orbits on $\mathcal{S}'_\varepsilon$. Moreover, whenever one of the original limit sets was hyperbolic relative to $\mathcal{S}$, we can ensure that the lifted limit set is of the same linear stability type relative to $\mathcal{S}'_\varepsilon$. This completes the proof. \end{proof}

Several remarks are in order.

\begin{rem}[Mass action kinetics]
The result clearly holds if we insist that both $\mathcal{R}$ and $\mathcal{R}'$ have mass action kinetics which is simply a special case of fixed power-law kinetics. In this case, in the proof of Theorem~\ref{mainthm}, $\alpha_j$ is the stoichiometric coefficient of the new species in the reactant complex of the $j$th reaction of $\mathcal{R}'$.
\end{rem}

\begin{rem}[The proof is constructive]
The proof of Theorem~\ref{mainthm}, as with other inheritance results based on perturbation theory, is constructive. It tells us how to set rate constants and how to choose a stoichiometric class in order to find the desired behaviour in the enlarged CRN $\mathcal{R}'$.
\end{rem}

\begin{rem}[The projected dynamics are close to the original]
Consider some ``lifted'' bounded orbit of $\mathcal{R}'$ such as $\mathcal{O}'_\varepsilon$ in the proof of Theorem~\ref{mainthm}. Its projection $\mathcal{O}_\varepsilon$ onto $x$ coordinates, can be made as close as we desire to the original orbit of $\mathcal{R}$ (namely, $\mathcal{O}$) by choosing $\varepsilon$ to be small. But this comes at the cost of large values of the new species concentration $y$ on the lifted orbit, and small rate constants. The next remark indicates the limitations of the lifting process. 
\end{rem}

\begin{rem}[We cannot always control the lifted dynamics over an entire stoichiometric class]
\label{remlift}
The proof of Theorem~\ref{mainthm} tells us the following: given any positive stoichiometric class, say $\mathcal{S}_+$, of the original CRN $\mathcal{R}$, fixing the perturbation parameter $\varepsilon$ at any positive value selects a positive stoichiometric class, say $(\mathcal{S}'_\varepsilon)_+$, of the lifted CRN $\mathcal{R}'$. Assume that rate constants are fixed and let $V$ and $V'_\varepsilon$ refer to the original and lifted vector fields on $\mathcal{S}_+$ and $(\mathcal{S}'_\varepsilon)_+$ respectively. Choosing $\varepsilon$ to be small ensures that the projection of $V'_\varepsilon$ onto $\mathcal{S}_+$ is close to the original vector field $V$ on $\mathcal{S}_+$ on that portion of $\mathcal{S}'_\varepsilon$ where the concentration of the added species (denoted by $y$ in the proof) is large. But, regardless of how small we choose $\varepsilon$ to be, if there are regions of $(\mathcal{S}'_\varepsilon)_+$ where $y$ is small, then in these regions $V'_\varepsilon$ need not be close to $V$. The consequences are illustrated in the example of the Brusselator in Section~\ref{subsec:ex_brusselator}, where the lifting process leads to a loss of permanence on every stoichiometric class. Note, however, that if $\mathcal{S}_+$, the original positive stoichiometric class of $\mathcal{R}$, is itself bounded, then we can control the lifted vector field over the entirety of $(\mathcal{S}'_\varepsilon)_+$.
\end{rem}

\begin{rem}[Normally hyperbolic invariant manifolds persist]
Although Theorem~\ref{mainthm} is phrased in terms of equilibria and periodic orbits, the result admits generalisation. Indeed, with the assumptions of the theorem, if $\mathcal{O}$ is any positive, compact, invariant manifold admitted by $\mathcal{R}$ and normally hyperbolic relative to its stoichiometric class $\mathcal{S}$, then it survives $C^1$ perturbations \cite{Fenichel71,HPS77}, and hence is admitted by $\mathcal{R}'$. If, for example, $\mathcal{R}$ admits a $k$-dimensional torus on some positive stoichiometric class, and the torus is normally hyperbolic relative to this class, then the same holds for $\mathcal{R}'$. 
\end{rem}

\begin{rem}[Bifurcations persist]
\label{rembif}
Suppose $\mathcal{R}$ admits, on some positive stoichiometric class, a nondegenerate local bifurcation of an equilibrium or periodic orbit, unfolded nondegenerately by the rate constants. Then, for sufficiently small, fixed, $\varepsilon > 0$, $\mathcal{R}'$ admits the same nondegenerate local bifurcation on some positive stoichiometric class as we vary the same combination of rate constants (note that rate constants of $\mathcal{R}$ and $\mathcal{R}'$ are in natural one-to-one correspondence). Essentially, the nondegeneracy and transversality conditions associated with the bifurcation, allow us to continue the bifurcation as we vary $\varepsilon$. Moreover, these conditions continue to hold for sufficiently small $\varepsilon$. For a concrete example demonstrating the pesistence of a bifurcation, see Section~\ref{subsec:ex_lva}.
\end{rem}

\begin{rem}[A generalisation of previous claims]
For CRNs with power-law kinetics, Theorem~\ref{mainthm} generalises the claims in Theorems~3 in \cite{banajipanteaMPNE} and \cite{banajiCRNosci}, which treat the very special case where the added species figures only trivially in reactions, i.e., adds only a row of zeros to the stoichiometric matrix of the network. Note that, in that case, the new stoichiometric classes were bounded if and only if the original stoichiometric classes were bounded. 
\end{rem}

\begin{rem}[Generalisations to other classes of kinetics]
Phrasing the result in terms of power-law kinetics simplifies the proof, but is not key to it. The broad template of the proof can be applied to CRNs with other classes of kinetics. 
\end{rem}

\begin{rem}[CRNs with bounded stoichiometric classes do not have greatly restricted dynamics]
One immediate consequence of Theorem~\ref{mainthm} is that insisting a CRN has bounded stoichiometric classes, does not greatly restrict its behaviour. If a given CRN with unbounded stoichiometric classes admits some finite set of bounded nondegenerate limit sets on one of its stoichiometric classes, then we can always construct, by adding in a dependent species, a CRN with bounded stoichiometric classes which admits the same bounded nondegenerate limit sets on one of its stoichiometric classes. We see several instances of this in the examples presented in Section~\ref{secexamples}.
\end{rem}

The next corollary tells us that it may be helpful to examine full-rank subnetworks of a CRN: finding nontrivial behaviours in these subnetworks is sufficient to ensure that they occur in the original CRN. We recall the definition of an induced subnetwork of a CRN from \cite{banajiCRNosci}: this is a CRN obtained by removing some reactions from a CRN, and/or some species from all the reactions in which they figure. In terms of the Petri-net graph of the CRN this corresponds to removing some vertices from the graph along with all their incident arcs. 

\begin{cor}
\label{corrank}
Let $\mathcal{R}$ be a CRN of rank $r$, and let $\mathcal{R}_0$ be any rank-$r$ induced subnetwork of $\mathcal{R}$. If, for some fixed power-law kinetics, $\mathcal{R}_0$ admits $k_1$ positive nondegenerate (resp., hyperbolic, resp., linearly stable) equilibria and $k_2$ positive nondegenerate (resp., hyperbolic, resp., linearly stable) periodic orbits, then the same holds for $\mathcal{R}$ with any derived power-law kinetics.
\end{cor}
\begin{proof}
Clearly $\mathcal{R}$ can be built from $\mathcal{R}_0$ by adding linearly dependent reactions and linearly dependent species to $\mathcal{R}_0$. The result is thus an immediate consequence of Theorem~\ref{mainthm} above which deals with the case of adding linearly dependent species, and Theorems~1 in \cite{banajipanteaMPNE}~and~\cite{banajiCRNosci} which deal with the case of adding linearly dependent reactions. (Although Theorems~1 in \cite{banajipanteaMPNE}~and~\cite{banajiCRNosci} are stated in more restricted terms, the generalisations required are immediate.)
\end{proof}

\section{Examples}

\label{secexamples}

First, we introduce some terminology relevant to the examples below.

{\bf Homogeneous CRNs.} The \emph{molecularity} of a complex in a CRN is the sum of its stoichiometric coefficients. We call a CRN \emph{homogeneous} if, for every reaction, the molecularities of the reactant complex and the product complex are equal. Clearly, this condition is equivalent to $(1,1,\ldots,1)^\mathrm{t}$ being an element of the kernel of $\Gamma^\mathrm{t}$. In particular, the stoichiometric classes of a homogeneous CRN are bounded: for the ODE associated with a homogeneous CRN, regardless of the precise nature of the kinetics, the quantity $x_1 + x_2 + \cdots + x_n$ is conserved. A partial converse is also true: it is easily shown that if a CRN is endowed with any fixed power-law kinetics, and $x_1 + x_2 + \cdots + x_n$ is constant along trajectories for some open set of rate constants, then the CRN is homogeneous.

{\bf Homogenisation of CRNs.} Starting with an arbitrary CRN $\mathcal{R}$, one can make it homogeneous by adding a new species with appropriate stoichiometric coefficients to the reactant or product complex of each reaction \cite[Exercise 4 on page 29]{erdi:toth:1989}. This operation preserves the rank of $\mathcal{R}$, and the homogenisation can be carried out in multiple ways. By Theorem~\ref{mainthm}, if $\mathcal{R}$ has mass action kinetics, and nondegenerate multistationarity (resp., oscillation) occurs in $\mathcal{R}$, then it also occurs in the homogenised CRN.

In all of the examples below, we start with a network that is not homogeneous, and then homogenise it. The examples in \Cref{subsec:ex_scalar,subsec:ex_lva} illustrate our main result, while those in \Cref{subsec:ex_lotka,subsec:ex_brusselator} demonstrate its limitations. The rank of each of these networks is one or two. For the homogenisation of some rank-three mass-action systems, consult \cite{boros:hofbauer:2022b}. For applications of Theorem~\ref{mainthm} where the enlarged network is not homogeneous, but nevertheless has bounded stoichiometric classes, see \cite{boros:hofbauer:2022a}.

Mathematica code for the analysis of all the examples is available on GitHub \cite{balazsgithub}.

\subsection{Schl\"ogl model: a single-species CRN with multiple equilibria}
\label{subsec:ex_scalar}

We use the reversible version of the Schl\"ogl model \cite{schlogl:1971,schlogl:1972} to demonstrate the use of Theorem~\ref{mainthm} for guaranteeing the existence of multiple nondegenerate equilibria in an enlarged reaction network. It also provides some insight into the proof of Theorem~\ref{mainthm}.

Consider the following single-species mass action system and its associated differential equation:
\begin{center}
\begin{tikzpicture}[scale=1.5]
\node (P0) at (0,0) {$\mathsf{0}$};
\node (P1) at (1,0) {$\mathsf{X}$};
\node (P2) at (2,0) {$2\mathsf{X}$};
\node (P3) at (3,0) {$3\mathsf{X}$};
\draw[arrows={-stealth},transform canvas={yshift=2pt}] (P0) to node[above] {$6$} (P1);
\draw[arrows={-stealth},transform canvas={yshift=-2pt}] (P1) to node[below] {$11$} (P0);
\draw[arrows={-stealth},transform canvas={yshift=2pt}] (P2) to node[above] {$6$} (P3);
\draw[arrows={-stealth},transform canvas={yshift=-2pt}] (P3) to node[below] {$1$} (P2);
\node at (5.5,0) {$\dot{x} = -x^3+6x^2-11x+6.$};
\end{tikzpicture}
\end{center}
It has $3$ positive nondegenerate equilibria: at $x=1$, $x=2$, and $x=3$. The first and third are linearly stable, while the one at $x=2$ is linearly unstable. 

We now homogenise the network. Note that the simplest choice of homogenisation gives the enlarged network $\mathsf{Y}\rightleftharpoons\mathsf{X},\,\, 2\mathsf{X}+\mathsf{Y} \rightleftharpoons 3\mathsf{X}$. However, we choose the slightly more complicated homogenised network $2\mathsf{Y} \rightleftharpoons \mathsf{X}+\mathsf{Y}, \,\, 2\mathsf{X}+\mathsf{Y} \rightleftharpoons 3\mathsf{X}$ in order to demonstrate two points: that there are many ways to homogenise a CRN; and that the modified rate constants given in the proof of Theorem~\ref{mainthm} can have nonlinear dependence on a perturbation parameter $\varepsilon$ as seen below.

Theorem~\ref{mainthm} now tells us that this network must admit three nondegenerate equilibria on some stoichiometric class, two linearly stable and one linearly unstable. To see why, we follow the proof of Theorem~\ref{mainthm} and obtain the following mass action system and its associated differential equation, dependent on a new parameter $\varepsilon$:
\begin{center}
\begin{tikzpicture}[scale=1.5]
\node (P0) at (-1,0) {$2\mathsf{Y}$};
\node (P1) at (0.25,0) {$\mathsf{X}+\mathsf{Y}$};
\node (P2) at (1.5,0) {$2\mathsf{X}+\mathsf{Y}$};
\node (P3) at (2.75,0) {$3\mathsf{X}$};
\draw[arrows={-stealth},transform canvas={yshift=2pt}] (P0) to node[above] {$6\varepsilon^2$} (P1);
\draw[arrows={-stealth},transform canvas={yshift=-2pt}] (P1) to node[below] {$11\varepsilon$} (P0);
\draw[arrows={-stealth},transform canvas={yshift=2pt}] (P2) to node[above] {$6\varepsilon$} (P3);
\draw[arrows={-stealth},transform canvas={yshift=-2pt}] (P3) to node[below] {$1$} (P2);
\node at (5.5,0) {$\begin{aligned}
\dot{x} &= -x^3+6\varepsilon x^2 y-11\varepsilon x y+6\varepsilon^2 y^2, \\
\dot{y} &= +x^3-6\varepsilon x^2 y+11\varepsilon x y-6\varepsilon^2 y^2.
\end{aligned}$};
\end{tikzpicture}
\end{center}

In the homogenised system, the stoichiometric subspace remains $1$-dimensional, and the quantity $x+y$ is conserved. We now restrict attention to the stoichiometric class defined by $x+y=\frac{1}{\varepsilon}$ and replace $y$ by $\frac{1}{\varepsilon}- x$. The dynamics of $x$ for $0 < x < \frac{1}{\varepsilon}$ is then given by
\begin{align*}
\dot x = -x^3 + 6x^2(1-\varepsilon x) - 11x(1-\varepsilon x) + 6(1-\varepsilon x)^2\,.
\end{align*}
On any compact subinterval of $(0,\infty)$, the vector field $-x^3+6x^2(1-\varepsilon x)-11x(1-\varepsilon x)+6(1-\varepsilon x)^2$ converges uniformly to $-x^3+6x^2-11x+6$ as $\varepsilon\to 0$. It is not hard to see that for all sufficiently small $\varepsilon$ the lifted system with the scaled rate constants has $3$ positive equilibria in the stoichiometric class $x+y=\frac{1}{\varepsilon}$, two of which are stable, while one is unstable. Note, however, that for any fixed rate constants, each stoichiometric class defined by $x+y=c$ with $c>0$ being large enough, has a unique positive equilibrium. This illustrates that in order to obtain the desired behaviour we must {\em simultaneously} choose the rate constants and a stoichiometric class of the lifted system.


\subsection{Lotka-Volterra-Autocatalator}
\label{subsec:ex_lva}

Consider the Lotka-Volterra-Autocatalator \cite[(8)]{farkas:noszticzius:1985} with one reversible reaction \cite[Section 3]{simon:1992} (henceforth abbreviated to LVA), and its associated mass action differential equation:
\begin{center}
\begin{tikzpicture}[scale=1.5]
\node (P1a) at (0,0)    {$2\mathsf{X}$};
\node (P1b) at (1,0)    {$3\mathsf{X}$};
\node (P2a) at (0,-2/3) {$\mathsf{X}+\mathsf{Y}$};
\node (P2b) at (1,-2/3) {$2\mathsf{Y}$};
\node (P3a) at (0,-4/3) {$\mathsf{Y}$};
\node (P3b) at (1,-4/3) {$\mathsf{0}$};
\draw[arrows={-stealth},transform canvas={yshift=2pt}] (P1a) to node[above] {$\kappa_1$} (P1b);
\draw[arrows={-stealth},transform canvas={yshift=-2pt}] (P1b) to node[below] {$\kappa_2$} (P1a);
\draw[arrows={-stealth}] (P2a) to node[above] {$\kappa_3$} (P2b);
\draw[arrows={-stealth}] (P3a) to node[above] {$\kappa_4$} (P3b);
\node at (3,-2/3) {$\begin{aligned}
\dot{x} &= \kappa_1 x^2-\kappa_2 x^3 - \kappa_3 x y, \\
\dot{y} &= \kappa_3 x y - \kappa_4 y.
\end{aligned}$};
\end{tikzpicture}
\end{center}
The system has exactly one positive equilibrium if and only if $\kappa_1 \kappa_3 > \kappa_2 \kappa_4$ (otherwise there is no positive equilibrium). We can confirm that at $\kappa_1 \kappa_3 = 2\kappa_2 \kappa_4$ the positive equilibrium undergoes a supercritical Andronov-Hopf bifurcation, and thus, for $\kappa_1 \kappa_3 - 2\kappa_2 \kappa_4$ positive and sufficiently small, the system has a stable periodic orbit.

Let us now homogenise the LVA. The resulting network and its associated mass action differential equation take the form
\begin{center}
\begin{tikzpicture}[scale=1.5]
\node (P1a) at (0,0)    {$2\mathsf{X}+\mathsf{Z}$};
\node (P1b) at (1,0)    {$3\mathsf{X}$};
\node (P2a) at (0,-2/3) {$\mathsf{X}+\mathsf{Y}$};
\node (P2b) at (1,-2/3) {$2\mathsf{Y}$};
\node (P3a) at (0,-4/3) {$\mathsf{Y}$};
\node (P3b) at (1,-4/3) {$\mathsf{Z}$};
\draw[arrows={-stealth},transform canvas={yshift=2pt}] (P1a) to node[above] {$\kappa_1$} (P1b);
\draw[arrows={-stealth},transform canvas={yshift=-2pt}] (P1b) to node[below] {$\kappa_2$} (P1a);
\draw[arrows={-stealth}] (P2a) to node[above] {$\kappa_3$} (P2b);
\draw[arrows={-stealth}] (P3a) to node[above] {$\kappa_4$} (P3b);
\node at (3,-2/3) {$\label{pp}\begin{aligned}
\dot{x} &= \kappa_1 x^2z-\kappa_2 x^3 - \kappa_3 x y, \\
\dot{y} &= \kappa_3 x y - \kappa_4 y, \\
\dot{z} &= \kappa_2 x^3 - \kappa_1 x^2 z + \kappa_4 y.
\end{aligned}$};
\end{tikzpicture}
\end{center}
Theorem~\ref{mainthm} tells us that the homogenised LVA must admit a linearly stable periodic orbit on some stoichiometric class. Let us use these systems to demonstrate the arguments in the proof of Theorem~\ref{mainthm} and some subsequent remarks. Consider again the ODE systems associated with the LVA,
\begin{equation}
\label{eqLVA}
\begin{aligned}
\dot{x} &= \kappa_1 x^2-\kappa_2 x^3 - \kappa_3 x y, \\
\dot{y} &= \kappa_3 x y - \kappa_4 y,
\end{aligned}
\end{equation}
and with the homogenised LVA, where we have now added a prime $'$ to the rate constants to distinguish them from those of the original LVA,
\begin{equation}
\label{eqLVAhom}
\begin{aligned}
\dot{x} &= \kappa'_1 x^2z-\kappa'_2 x^3 - \kappa'_3 x y, \\
\dot{y} &= \kappa'_3 x y - \kappa'_4 y, \\
\dot{z} &= \kappa'_2 x^3 - \kappa'_1 x^2 z + \kappa'_4 y.
\end{aligned}
\end{equation}
If we set $(\kappa'_1, \kappa'_2, \kappa'_3, \kappa'_4) = (\varepsilon \kappa_1, \kappa_2, \kappa_3, \kappa_4)$, and restrict attention to the stoichiometric class on which $x+y+z=\frac{1}{\varepsilon}$, then, using this equation to eliminate $z$, we find that $x$ and $y$ in the homogenised LVA evolve according to
\begin{equation}
\label{eqLVAred}
\begin{aligned}
\dot{x} &= \kappa_1 x^2-\kappa_2 x^3 - \kappa_3 x y -\varepsilon\kappa_1(x^3+x^2y), \\
\dot{y} &= \kappa_3 x y - \kappa_4 y\,.
\end{aligned}
\end{equation}
We see immediately that the vector field in (\ref{eqLVAred}) converges on compact sets to that in (\ref{eqLVA}) as $\varepsilon \to 0$.

In  particular, suppose we fix some values of $\kappa_1, \kappa_2, \kappa_3, \kappa_4$ such that $\kappa_1 \kappa_3 - 2\kappa_2 \kappa_4$ is positive and sufficiently small to ensure that (\ref{eqLVA}) has a linearly stable, positive periodic orbit $\mathcal{O}$. Then, for sufficiently small $\varepsilon > 0$, (\ref{eqLVAred}) has a linearly stable, positive, periodic orbit $\mathcal{O}_\varepsilon$. Moreover, provided $\varepsilon$ is sufficiently small, $z=\frac{1}{\varepsilon}-x-y$ remains positive as $x$ and $y$ vary along this periodic orbit, and so the lifted system (\ref{eqLVAhom}) has a positive periodic orbit $\mathcal{O}'_\varepsilon$ on the stoichiometric class defined by $x+y+z = \frac{1}{\varepsilon}$. Clearly, since $\mathcal{O}_\varepsilon$ is linearly stable for (\ref{eqLVAred}), $\mathcal{O}'_\varepsilon$ is linearly stable relative to this class.

In fact, we can go further and argue that the bifurcation which gives rise to stable periodic orbits in (\ref{eqLVA}) itself survives lifting (see Remark~\ref{rembif}). Suppose we fix some path in parameter space, $P\colon \mathbb{R} \to \mathbb{R}^4_+$, $s \mapsto (\kappa_1(s), \kappa_2(s), \kappa_3(s), \kappa_4(s))$ such that $P(0)$ lies on the surface in $\mathbb{R}^4_+$ defined by $\kappa_1 \kappa_3 - 2\kappa_2 \kappa_4=0$, and $P$ is transverse to this surface at $0$. This ensures that (\ref{eqLVA}) undergoes a nondegenerate supercritical Andronov-Hopf bifurcation at $s=0$. Then we can be sure that for any sufficiently small $\varepsilon > 0$, a nondegenerate supercritical Andronov-Hopf bifurcation occurs in (\ref{eqLVAhom}) on the stoichiometric class defined by $x+y+z=\frac{1}{\varepsilon}$ as we traverse the corresponding path $P'$ in $(\kappa'_1, \kappa'_2, \kappa'_3, \kappa'_4)$ space obtained by noting that $(\kappa'_1, \kappa'_2, \kappa'_3, \kappa'_4) = (\varepsilon \kappa_1, \kappa_2, \kappa_3, \kappa_4)$. This bifurcation occurs for some value of $s$ close to $0$.


\subsection{Lotka reactions}
\label{subsec:ex_lotka}

We now provide an example where the assumptions of Theorem~\ref{mainthm} are violated, and periodic solutions are not preserved by lifting.

Consider the Lotka reactions and the associated mass action differential equation:
\begin{center}
\begin{tikzpicture}[scale=1.5]
\node (P1a) at (0,0)    {$\mathsf{X}$};
\node (P1b) at (1,0)    {$2\mathsf{X}$};
\node (P2a) at (0,-2/3) {$\mathsf{X}+\mathsf{Y}$};
\node (P2b) at (1,-2/3) {$2\mathsf{Y}$};
\node (P3a) at (0,-4/3) {$\mathsf{Y}$};
\node (P3b) at (1,-4/3) {$\mathsf{0}$};
\draw[arrows={-stealth}] (P1a) to node[above] {$\kappa_1$} (P1b);
\draw[arrows={-stealth}] (P2a) to node[above] {$\kappa_2$} (P2b);
\draw[arrows={-stealth}] (P3a) to node[above] {$\kappa_3$} (P3b);
\node at (3,-2/3) {$\begin{aligned}
\dot{x} &= \kappa_1 x - \kappa_2 x y, \\
\dot{y} &= \kappa_2 x y - \kappa_3 y.
\end{aligned}$};
\end{tikzpicture}
\end{center}
In this case, the unique positive equilibrium $\left(\frac{\kappa_3}{\kappa_2},\frac{\kappa_1}{\kappa_2}\right)$ is surrounded by a continuum of periodic orbits; these are level sets of the nonlinear first integral $x^{-\kappa_3}y^{-\kappa_1}e^{\kappa_2(x+y)}$. Since these periodic orbits are degenerate, Theorem~\ref{mainthm} does not apply. 

Indeed, adding a new species, $\mathsf{Z}$, to some of the reactions while preserving the rank of the network can lead to the destruction of all of the periodic orbits. Consider the following network and its associated mass action differential equation:
\begin{center}
\begin{tikzpicture}[scale=1.5]
\node (P1a) at (0,0)    {$\mathsf{X}+\mathsf{Z}$};
\node (P1b) at (1,0)    {$2\mathsf{X}$};
\node (P2a) at (0,-2/3) {$\mathsf{X}+\mathsf{Y}$};
\node (P2b) at (1,-2/3) {$2\mathsf{Y}$};
\node (P3a) at (0,-4/3) {$\mathsf{Y}$};
\node (P3b) at (1,-4/3) {$\mathsf{Z}$};
\draw[arrows={-stealth}] (P1a) to node[above] {$\kappa_1$} (P1b);
\draw[arrows={-stealth}] (P2a) to node[above] {$\kappa_2$} (P2b);
\draw[arrows={-stealth}] (P3a) to node[above] {$\kappa_3$} (P3b);
\node at (3,-2/3) {$\begin{aligned}
\dot{x} &= \kappa_1 x z - \kappa_2 x y, \\
\dot{y} &= \kappa_2 x y - \kappa_3 y,   \\
\dot{z} &= -\kappa_1 x z + \kappa_3 y.
\end{aligned}$};
\end{tikzpicture}
\end{center}
The set of positive equilibria is $\left\{\left(\frac{\kappa_3}{\kappa_2},\frac{\kappa_1}{\kappa_2}t,t\right)\colon t>0\right\}$. Thus, the stoichiometric classes $x+y+z=c$ with $c\leq \frac{\kappa_3}{\kappa_2}$ have no positive equilibria, while those with $c> \frac{\kappa_3}{\kappa_2}$ have a unique positive equilibrium. Note that the divergence of the vector field after division by $xyz$ equals $-\frac{\kappa_3}{xz^2}$. Since this quantity is negative on $\mathbb{R}^3_+$, there is no periodic orbit that lies entirely in the positive orthant \cite[Satz 1]{schneider:1969} (see also Remark (v) following the proof of Theorem 2.3 in \cite{li_1996}). In fact, every positive equilibrium is globally asymptotically stable within its positive stoichiometric class: this follows via the Poincar\'e-Bendixson Theorem, since it can be shown that on any stoichiometric class which includes a positive equilibrium no positive initial condition has omega limit set intersecting the boundary of the nonnegative orthant.


\subsection{Brusselator}
\label{subsec:ex_brusselator}

Our final example demonstrates that while the lifted CRN must admit the nondegenerate behaviours of the original CRN, it may also allow other behaviours not seen in the original CRN, such as multiple periodic orbits and homoclinic orbits. Looked at from another angle, omitting a single species from a CRN, even without changing its rank, can result in the loss of many different nontrivial behaviours.

Consider the Brusselator and its associated mass action differential equation:
\begin{center}
\begin{tikzpicture}[scale=1.5]
\node (P1) at (0,0)    {$\mathsf{0}$};
\node (P2) at (1,0)    {$\mathsf{X}$};
\node (P3) at (2,0) {$\mathsf{Y}$};
\node (P4) at (1/3,-2/3) {$2\mathsf{X}+\mathsf{Y}$};
\node (P5) at (5/3,-2/3) {$3\mathsf{X}$};
\draw[arrows={-stealth},transform canvas={yshift=2pt}] (P1) to node[above] {$\kappa_1$} (P2);
\draw[arrows={-stealth},transform canvas={yshift=-2pt}] (P2) to node[below] {$\kappa_2$} (P1);
\draw[arrows={-stealth}] (P2) to node[above] {$\kappa_3$} (P3);
\draw[arrows={-stealth}] (P4) to node[above] {$\kappa_4$} (P5);
\node at (4,-1/3) {$\begin{aligned}
\dot{x} &= \kappa_1 -\kappa_2 x - \kappa_3 x + \kappa_4 x^2y, \\
\dot{y} &= \kappa_3 x - \kappa_4 x^2y.
\end{aligned}$};
\end{tikzpicture}
\end{center}
At any fixed values of the rate constants the system has a unique positive equilibrium, $(x^*,y^*)=\left(\frac{\kappa_1}{\kappa_2},\frac{\kappa_2 \kappa_3}{\kappa_1 \kappa_4}\right)$. It can be shown that the system is permanent. Moreover, it is known that $(x^*,y^*)$ is globally asymptotically stable for $\kappa_3 \leq \kappa_2 + \frac{\kappa_1^2 \kappa_4}{\kappa_2^2}$, while it is repelling for $\kappa_3 > \kappa_2 + \frac{\kappa_1^2 \kappa_4}{\kappa_2^2}$ and is surrounded by a periodic orbit that is born via a supercritical Andronov-Hopf bifurcation. Moreover, this periodic orbit is unique and attracts every positive initial condition, except $(x^*,y^*)$ (see \cite[Example 5 on page 135]{ye:1986}). 

Let us now homogenise the Brusselator. The resulting CRN and its mass action differential equation take the form
\begin{center}
\begin{tikzpicture}[scale=1.5]
\node (P1) at (0,0)    {$\mathsf{Z}$};
\node (P2) at (1,0)    {$\mathsf{X}$};
\node (P3) at (2,0) {$\mathsf{Y}$};
\node (P4) at (1/3,-2/3) {$2\mathsf{X}+\mathsf{Y}$};
\node (P5) at (5/3,-2/3) {$3\mathsf{X}$};
\draw[arrows={-stealth},transform canvas={yshift=2pt}] (P1) to node[above] {$\kappa_1$} (P2);
\draw[arrows={-stealth},transform canvas={yshift=-2pt}] (P2) to node[below] {$\kappa_2$} (P1);
\draw[arrows={-stealth}] (P2) to node[above] {$\kappa_3$} (P3);
\draw[arrows={-stealth}] (P4) to node[above] {$\kappa_4$} (P5);
\node at (4,-1/3) {$\begin{aligned}
\dot{x} &= \kappa_1 z -\kappa_2 x - \kappa_3 x + \kappa_4 x^2y, \\
\dot{y} &= \kappa_3 x - \kappa_4 x^2y, \\
\dot{z} &= -\kappa_1 z + \kappa_2 x.
\end{aligned}$};
\end{tikzpicture}
\end{center}
By Theorem~\ref{mainthm}, there exist rate constants and a stoichiometric class where the new system has a stable periodic orbit. In fact, given any rate constants satisfying $\kappa_3 > \kappa_2 + \frac{\kappa_1^2 \kappa_4}{\kappa_2^2}$ and such that the Brusselator has a linearly stable periodic orbit, the proof of Theorem~\ref{mainthm} gives us a procedure for finding a linearly stable periodic orbit in the lifted system. Moreover, by Remark~\ref{rembif}, a supercritical Andronov-Hopf bifurcation must occur on some stoichiometric class as rate constants are varied in the lifted system.

Note, however, that the global behaviour of the homogenised Brusselator is quite different from that of the Brusselator. The first interesting difference is that whereas the original system was permanent, the lifted system is not permanent on any stoichiometric class intersecting the positive orthant; we show in \ref{sec:app_brusselator} that for all rate constants and all positive values of the parameter $c$, the stoichiometric class satisfying $x+y+z=c$ includes a boundary equilibrium $(0,c,0)$ which is asymptotically stable relative to this class. Thus we cannot control the vector field at all points on any of these stoichiometric classes, even though these classes are compact (see Remark~\ref{remlift} following the proof of Theorem~\ref{mainthm}).

Aside from the loss of permanence, the homogeneous Brusselator admits several nontrivial behaviours forbidden in the original. In \ref{sec:app_brusselator} we show that for the homogenised network both supercritical and subcritical Andronov-Hopf bifurcations can occur on some stoichiometric classes as we vary rate constants. Furthermore, various interesting co-dimension two bifurcations take place: a generic Bautin bifurcation \cite[Section 8.3]{kuznetsov:2004} and a generic Bogdanov-Takens bifurcation \cite[Section 8.4]{kuznetsov:2004} can both occur. This implies that the lifted system admits, for various choices of the rate constants, the following behaviours not seen in the original Brusselator:
\begin{itemize}
\item a fold bifurcation of equilibria;
\item an unstable positive equilibrium surrounded by a stable and an unstable periodic orbit;
\item a fold bifurcation of periodic orbits;
\item a stable positive equilibrium surrounded by an unstable periodic orbit or a homoclinic orbit;
\item a homoclinic bifurcation.
\end{itemize}
It is not unexpected that the lifted system admits richer dynamical behaviour; but it is worth noting that if we restrict attention to any stoichiometric class and consider the evolution of concentrations of $\mathsf{X}$ and $\mathsf{Y}$, then the (2-dimensional) lifted vector field differs from the original only by the addition of linear terms (see equation (\ref{eq:ode_lifted_brusselator_2d}) in \ref{sec:app_brusselator}). The addition of linear terms to a 2D differential equation can thus quite dramatically increase the complexity of its behaviour.

\section{Discussion and conclusions}

Inheritance results tell us how we might enlarge a CRN while preserving its capacity for various dynamical behaviours. The main result in this paper adds a simple but important inheritance result relevant to the study of both multistationarity and oscillation in CRNs. Some inheritance results, including the main result of this paper, are gathered in \cite{banajisplitreacs}. Taken together, these results provide a powerful tool for predicting nontrivial behaviours in a CRN based on its subnetworks.

It is useful to think of inheritance results in terms of partial orders. For example, suppose a CRN $\mathcal{R}$ with mass action kinetics admits linearly stable oscillation. The same then holds for all CRNs greater than $\mathcal{R}$ in the partial order defined by available inheritance results. We may then look for CRNs admitting linearly stable oscillation which are minimal with respect to this partial order in order to gain insight into the capacity for stable oscillation in larger and more complex CRNs \cite{banajiCRNosci}. Such a program provides a rigorous basis for claims about ``motifs'', namely small subnetworks which are at the root of certain behaviours in biological systems. Systematically identifying minimal CRNs with prescribed behaviours, followed by the development of algorithms to test for their presence in larger CRNs, is a natural avenue for future work. 

While inheritance results are often phrased in terms of enlarging CRNs, they can also be seen in terms of modelling choices. The main result here is relevant to choices which might affect conservation laws. In physically realistic systems of chemical reactions occurring in a closed environment we expect numbers of atoms of each element to be conserved: this is the so-called {\em law of atomic balance} \cite{erdi:toth:1989}. However, it is common in modelling CRNs to omit some species from reactions, particularly when they are considered to be present in abundance, or their concentration is subject to external control. As an example, when reactions involving ATP and ADP occur in biochemical models, inorganic phosphate and water are often omitted from the equations. Such omissions tend to destroy physical conservation laws. 

It is natural to worry that omitting species from reactions might introduce fundamentally new, and unrealistic, behaviours into the system. Theorems~4 in \cite{banajipanteaMPNE} and \cite{banajiCRNosci} provide some reassurance that this will not occur if we omit species whose concentration is sufficiently strongly controlled by external processes. In a similar way, Theorem~\ref{mainthm} here provides some reassurance in the case where the omitted species are linearly dependent on the others. In this case, the omissions {\em cannot} introduce the capacity for behaviours such as stable oscillation or multistability. In fact, any behaviour of the simplified CRN occurring on a compact set, and which is robust in the sense that it survives $C^1$ perturbation, can also be obtained on some stoichiometric class of the larger CRN for appropriate choices of parameters. We can, of course, still {\em lose} interesting behaviours by omitting dependent species, as illustrated by the example of the Brusselator in Section~\ref{subsec:ex_brusselator}.

The examples demonstrated both the power and limitations of the result. Crucially, we only expect {\em nondegenerate} behaviours to survive the addition of new species as illustrated by the example in Section~\ref{subsec:ex_lotka}. And the result is not global: introducing new linearly dependent species may introduce new global behaviours regardless of how small we make the perturbation parameter $\varepsilon$. In particular, we saw, in the example of the Brusselator (Section~\ref{subsec:ex_brusselator}), how the lifted system could allow positive trajectories to converge to the boundary of state space, even though the original system had a positive, globally attracting, compact set for all choices of parameters. The details are laid out in \ref{sec:app_brusselator}.

\appendix


\section{Dynamics of the homogenised Brusselator}
\label{sec:app_brusselator}

Consider the homogenised Brusselator and its associated mass action differential equation:
\begin{center}
\begin{tikzpicture}[scale=1.5]
\node (P1) at (0,0)    {$\mathsf{Z}$};
\node (P2) at (1,0)    {$\mathsf{X}$};
\node (P3) at (2,0) {$\mathsf{Y}$};
\node (P4) at (1/3,-2/3) {$2\mathsf{X}+\mathsf{Y}$};
\node (P5) at (5/3,-2/3) {$3\mathsf{X}$};
\draw[arrows={-stealth},transform canvas={yshift=2pt}] (P1) to node[above] {$\kappa_1$} (P2);
\draw[arrows={-stealth},transform canvas={yshift=-2pt}] (P2) to node[below] {$\kappa_2$} (P1);
\draw[arrows={-stealth}] (P2) to node[above] {$\kappa_3$} (P3);
\draw[arrows={-stealth}] (P4) to node[above] {$\kappa_4$} (P5);
\node at (4,-1/3) {$\begin{aligned}
\dot{x} &= \kappa_1 z -\kappa_2 x - \kappa_3 x + \kappa_4 x^2y, \\
\dot{y} &= \kappa_3 x - \kappa_4 x^2y, \\
\dot{z} &= -\kappa_1 z + \kappa_2 x.
\end{aligned}$};
\end{tikzpicture}
\end{center}
The stoichiometric classes which intersect the positive orthant are $\mathcal{P}_c = \{(x,y,z) \in \mathbb{R}^3_{\geq0} \colon x+y+z=c\}$ for $c>0$. Note that for each positive $c$, the corner $(0,c,0)$ of the triangle $\mathcal{P}_c$ is an equilibrium of this system (and this is the only boundary equilibrium in $\mathcal{P}_c$).

\subsection{The homogenised Brusselator is not permanent}

Let us now restrict the dynamics to $\mathcal{P}_c$. After elimination of $z$ by the conservation law $x+y+z=c$, we obtain the following ODE system on $\{(x,y) \in \mathbb{R}^2_{\geq 0}\,:\, x+y \leq c\}$:
\begin{align} \label{eq:ode_lifted_brusselator_2d}
\begin{split}
\dot{x} &= \kappa_1 (c-x-y) -\kappa_2 x - \kappa_3 x + \kappa_4 x^2y, \\
\dot{y} &= \kappa_3 x - \kappa_4 x^2y.
\end{split}
\end{align}
From here on, we will focus on this system, which can be seen as describing, in local coordinates, the dynamics of the original ODEs restricted to a particular stoichiometric class, parameterised by $c$. Note that parameters of the system are now the four original rate constants $\kappa_1, \kappa_2, \kappa_3, \kappa_4$, along with $c$. For the purposes of bifurcation analysis we assume, however, that $c$ is fixed.

Observe that the Jacobian matrix of (\ref{eq:ode_lifted_brusselator_2d}) at the corner equilibrium $(0,c)$, denoted by $J_0$, equals
\begin{align*}
\begin{pmatrix}
-\kappa_1-\kappa_2-\kappa_3 & -\kappa_1 \\
\kappa_3                      & 0
\end{pmatrix}.
\end{align*}
Since $\det J_0 > 0$, $\tr J_0 < 0$, and
\begin{align*}
(\tr J_0)^2 - 4\det J_0 = (\kappa_1+\kappa_2+\kappa_3)^2 - 4 \kappa_1\kappa_3 = (\kappa_1+\kappa_2-\kappa_3)^2+4\kappa_2 \kappa_3>0,
\end{align*}
both eigenvalues are real and negative. Therefore, the boundary equilibrium $(0,c)$ is asymptotically stable. In particular, for all $(\kappa_1,\kappa_2,\kappa_3,\kappa_4,c)\in\mathbb{R}^5_+$, \eqref{eq:ode_lifted_brusselator_2d} is \emph{not} permanent, and since $c>0$ was arbitrary, the same clearly holds for the homogenised Brusselator on each positive stoichiometric class. By contrast, the Brusselator was permanent for all positive values of the rate constants.

\subsection{Local bifurcation analysis}

We proceed as follows. We first parameterise the set of equilibria which simplifies many calculations. The parameter used, denoted by $t$, can replace the parameter $c$ in many calculations: along the branch of equilibria we consider, $t$ and $c$ are in one-to-one correpondence. We write down necessary conditions for fold and Andronov-Hopf bifurcations to occur, and confirm that both supercritical and subcritical Andronov-Hopf bifurcations can occur. We also confirm that two codimension-$2$ bifurcations can occur on stoichiometric classes as we vary the rate constants: a generalised Andronov-Hopf bifurcation, also known as a Bautin bifurcation; and a Bogdanov-Takens bifurcation. We can check that apart from along exceptional sets, nondegeneracy and transversality conditions hold for all the bifurcations. See \cite{balazsgithub} for full details of all the calculations.

We first note that the set of positive equilibria of the homogenised Brusselator is
\begin{align*}
E_+ = \left\{\left(t,\frac{\kappa_3}{\kappa_4}\frac{1}{t},\frac{\kappa_2}{\kappa_1}t\right) \colon t>0 \right\}.
\end{align*}
The parameters $c$ and $t$ are related via $t\frac{\kappa_1+\kappa_2}{\kappa_1}+\frac{\kappa_3}{\kappa_4 t}=c$. Defining
\[
c^*:=2\sqrt{\frac{(\kappa_1+\kappa_2)\kappa_3}{\kappa_1 \kappa_4}},
\]
we find that $|E_+\cap\mathcal{P}_c| = 0$ if $c<c^*$, while $|E_+\cap\mathcal{P}_c| = 1$ if $c=c^*$, and $|E_+\cap\mathcal{P}_c| = 2$ if $c>c^*$. The equilbrium set intersects the stoichiometic class with parameter $c^*$ when $t=t^*$ where
\[
t^*:= \sqrt{\frac{\kappa_1\kappa_3}{(\kappa_1+\kappa_2)\kappa_4}}.
\]

The Jacobian matrix of \eqref{eq:ode_lifted_brusselator_2d} along $E_+$, denoted by $J(t)$, equals
\begin{align*}
\begin{pmatrix}
\kappa_3-\kappa_1-\kappa_2 & \kappa_4 t^2-\kappa_1 \\
-\kappa_3                  & -\kappa_4 t^2
\end{pmatrix}.
\end{align*}
Its determinant and trace equal
\[
\det J(t) = (\kappa_1+\kappa_2)\kappa_4 t^2 - \kappa_1 \kappa_3 \quad \mbox{and} \quad \tr J(t) = -\kappa_4 t^2 + (\kappa_3 - \kappa_1 - \kappa_2).
\]
In the positive stoichiometric classes $\mathcal{P}_c$ with $2$ positive equilibria, the equilibrium corresponding to the value of $t < t^*$, which is closer to the corner equilibrium $(0,c)$, is a saddle for all values of the rate constants. For the purposes of local bifurcation analysis, we thus focus attention on equilibria satisfying $t \geq t^*$. As values of $t$ and $c$ are in one-to-one correspondence for $t \geq t^*$ and $c \geq c^*$, we can thus pass back and forth between parameters $c$ and $t$, and do so, sometimes without comment, in calculations to follow. 

{\bf Fold bifurcations of equilibria.} These potentially occur when $\mathrm{det}\,J(t)=0$, namely on the set $\mathrm{T}$ in parameter space where $t=t^*$:
\[
\mathrm{T}:=\{(\kappa_1, \kappa_2, \kappa_3, \kappa_4, t)\in \mathbb{R}^5_+ \colon t^2=\frac{\kappa_1\kappa_3}{(\kappa_1+\kappa_2)\kappa_4}\}\,.
\]
Equivalently, fold bifurcations potentially occur when $c=c^*$, namely along:
\[
\widetilde{\mathrm{T}}:=\{(\kappa_1, \kappa_2, \kappa_3, \kappa_4, c)\in \mathbb{R}^5_+ \colon c^2= 4\frac{(\kappa_1+\kappa_2)\kappa_3}{\kappa_1\kappa_4}\}\,.
\]
For any fixed value of $c$, $\widetilde{\mathrm{T}}$ tells us combinations of rate constants for which we expect a fold bifurcation to occur on the stoichiometric class parameterised by $c$. It is straightforward to confirm that, provided $\kappa_3\neq(\kappa_1+\kappa_2)^2/\kappa_2$, the fold bifurcation is nondegenerate, and is unfolded nondegenerately by the rate constants $\kappa_i$ (any of $\kappa_1$, $\kappa_2$, $\kappa_3$ or $\kappa_4$ serves to unfold the bifurcation nondegenerately for all rate constants). The degenerate case where $\kappa_3 = (\kappa_1+\kappa_2)^2/\kappa_2$ will be of importance later.

{\bf Andronov-Hopf bifurcations.} These potentially occur when $\mathrm{tr}\,J(t) = 0$ and $\mathrm{det}\,J(t) > 0$, namely, when
\[
t^2=\frac{\kappa_3-\kappa_1-\kappa_2}{\kappa_4}, \quad t^2>\frac{\kappa_1\kappa_3}{(\kappa_1+\kappa_2)\kappa_4}\,.
\]
Combining these conditions, Andronov-Hopf bifurcations potentially occur along
\[
\mathrm{H}:=\{(\kappa_1, \kappa_2, \kappa_3, \kappa_4, t)\in \mathbb{R}^5_+ \colon t^2=\frac{\kappa_3-\kappa_1-\kappa_2}{\kappa_4},\,\quad \kappa_2\kappa_3>(\kappa_1+\kappa_2)^2\}\,.
\]
(We may, if desired, write the bifurcation set in terms of the parameters $\kappa_i$ and $c$ instead of $\kappa_i$ and $t$.) To check nondegeneracy of the Andronov-Hopf bifurcations, we compute the first focal value, $L_1$, along $\mathrm{H}$. Defining $a = \frac{\kappa_2}{\kappa_1}$ and $b = \frac{\kappa_3}{\kappa_1}$, we obtain $L_1 = \frac{1}{t^2}\frac{P(a,b)}{Q(a,b)}$, where
\begin{align*}
P(a,b)&=b^3[2-a] - b^2[5+3a-a^2] + b[5+12a+8a^2+a^3] - [4+13a+15a^2+7a^3+a^4],\\
Q(a,b)&=(ab - (1+a)^2)^{\frac32}(b-a-2).
\end{align*}

Note that the necessary condition for Andronov-Hopf bifurcation $\kappa_2\kappa_3 > (\kappa_1+\kappa_2)^2$ is equivalent to $ab > (1+a)^2$. Since $Q(a,b)$ is positive whenever this condition is satisfied, it suffices to investigate the sign of $P(a,b)$. As $a$ and $b$ vary along $\mathrm{H}$, $P(a,b)$, and hence $L_1$, can be positive, negative or zero along the bifurcation set, corresponding to subcritical, supercritical, and degenerate Andronov-Hopf bifurcations respectively, as shown in \Cref{fig:L1}.
\begin{figure}
\centering
\includegraphics[scale=0.6]{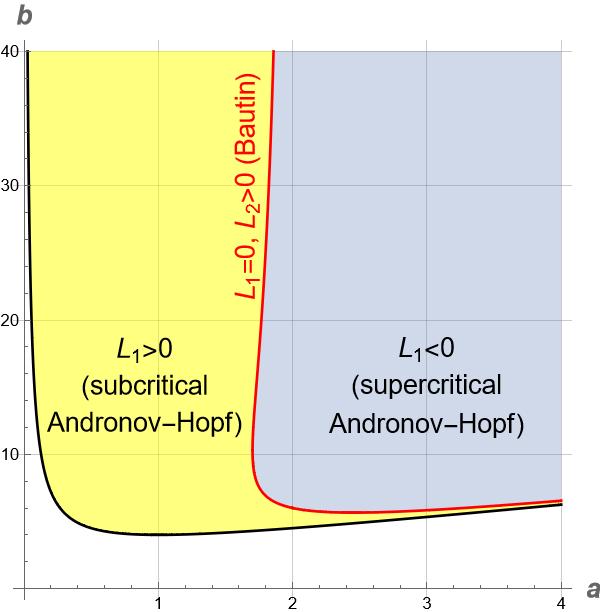}
\caption{Analysis of the Andronov-Hopf bifurcation for the mass-action system (\ref{eq:ode_lifted_brusselator_2d}) on $\mathrm{H}$. The figure shows the sign of the first focal value $L_1$ as a function of $a = \frac{\kappa_2}{\kappa_1}$ and $b = \frac{\kappa_3}{\kappa_1}$. The black curve is where $ab=(1+a)^2$, namely the boundary of $\mathrm{H}$ where the determinant of the Jacobian matrix vanishes and we expect Bogdanov-Takens bifurcations to occur; notice that it touches only the $L_1>0$ region, and we can indeed confirm that all Bogdanov-Takens bifurcations are subcritical. Along the red curve, where $L_1=0$ we can confirm that $L_2>0$. This is where we expect Bautin bifurcations to occur and, in fact, we find these bifurcations to be nondegenerate away from an exceptional set.}
\label{fig:L1}
\end{figure}

We can also confirm that wherever $L_1 \neq 0$, the parameters unfold the bifurcation nondegenerately.

{\bf Bautin bifurcations.} Generically, points along $\mathrm{H}$ where $L_1$ vanishes correspond to a generalised Andronov-Hopf bifurcation, also known as a Bautin bifurcation \cite{kuznetsov:2004}. We define the parameter set corresponding to potential Bautin bifurcations as
\[
\mathrm{GH} := \{(\kappa_1, \kappa_2, \kappa_3, \kappa_4, t) \subseteq \mathrm{H} \colon L_1 = 0\}\,.
\]
Indeed, the nondegeneracy conditions for the Bautin bifurcation can be checked and found to hold on $\mathrm{GH}$ apart from along an exceptional set:
\begin{enumerate}
\item Along $\mathrm{GH}$, we can compute the second focal value, $L_2$, and find it is always positive. The calculations are lengthy and are omitted, but can be found in the supporting documentation \cite{balazsgithub}.
\item The parameters $(\kappa_3, \kappa_4)$ unfold the bifurcation nondegenerately. Choosing these as our bifurcation parameters, then for fixed choices of the remaining parameters, we can confirm that the regularity condition (B.2) in \cite[Theorem 8.2]{kuznetsov:2004} is almost always satisfied on $\mathrm{GH}$, apart from along an exceptional set where $\kappa_2/\kappa_1=\overline{a}=\sqrt{\frac{73+22\sqrt{11}}{20}}-1\approx 1.70$. On \Cref{fig:L1}, a vertical line at $a=\overline{a}$ would touch the $L_1 = 0$ curve. The calculations are again omitted here but available in the supporting documentation \cite{balazsgithub}.
\end{enumerate}

Thus, away from the exceptional set, the Bautin bifurcation is nondegenerate. For parameter values near to $\mathrm{GH}$, on certain stoichiometric classes an unstable positive equilibrium is surrounded by a pair of periodic orbits, one stable and one unstable. As we vary rate constants, a fold bifurcation of periodic orbits must occur on some stoichiometric class. 

{\bf Bogdanov-Takens (B-T) bifurcations.} These potentially occur where $\mathrm{T}$ meets the closure of $\mathrm{H}$, namely along the set where $\mathrm{det}\,J(t)=0$ and $\mathrm{tr}\,J(t)=0$, defined by
\[
\mathrm{BT}:=\{(\kappa_1, \kappa_2, \kappa_3, \kappa_4, t)\in \mathbb{R}^5_+ \colon \kappa_2\kappa_3=(\kappa_1+\kappa_2)^2, \quad t^2=\frac{\kappa_1\kappa_3}{(\kappa_1+\kappa_2)\kappa_4}\}\,.
\]
We can confirm that the B-T bifurcation is always nondegerate, subcritical, and unfolded nondegenerately by parameters $(\kappa_3, \kappa_4)$. This involves checking conditions (BT.0), (BT.1), (BT.2), and (BT.3) in \cite[Theorem 8.4]{kuznetsov:2004}. The computations can be carried out explicitly, but are lengthy, and so are omitted here. They can be found in the supporting documentation \cite{balazsgithub}. Thus for parameter values close to $\mathrm{BT}$, on some stoichiometric classes, the system can have a stable equilibrium surrounded by an unstable periodic orbit. Furthermore, homoclinic bifurcations can occur.

In \Cref{fig:bifurcation_diagram}, we fix $\kappa_1 = 2$, $\kappa_2 = 4$, and $c = 6$, and depict the curves corresponding to $\mathrm{T}$ and $\mathrm{H}$, and the points corresponding to $\mathrm{BT}$ and $\mathrm{GH}$, in the $(\kappa_3, \kappa_4)$ plane. This completes our analysis of the local bifurcations of equilibria in the homogeneous Brusselator.
\begin{figure}
\centering
\includegraphics[scale=0.2]{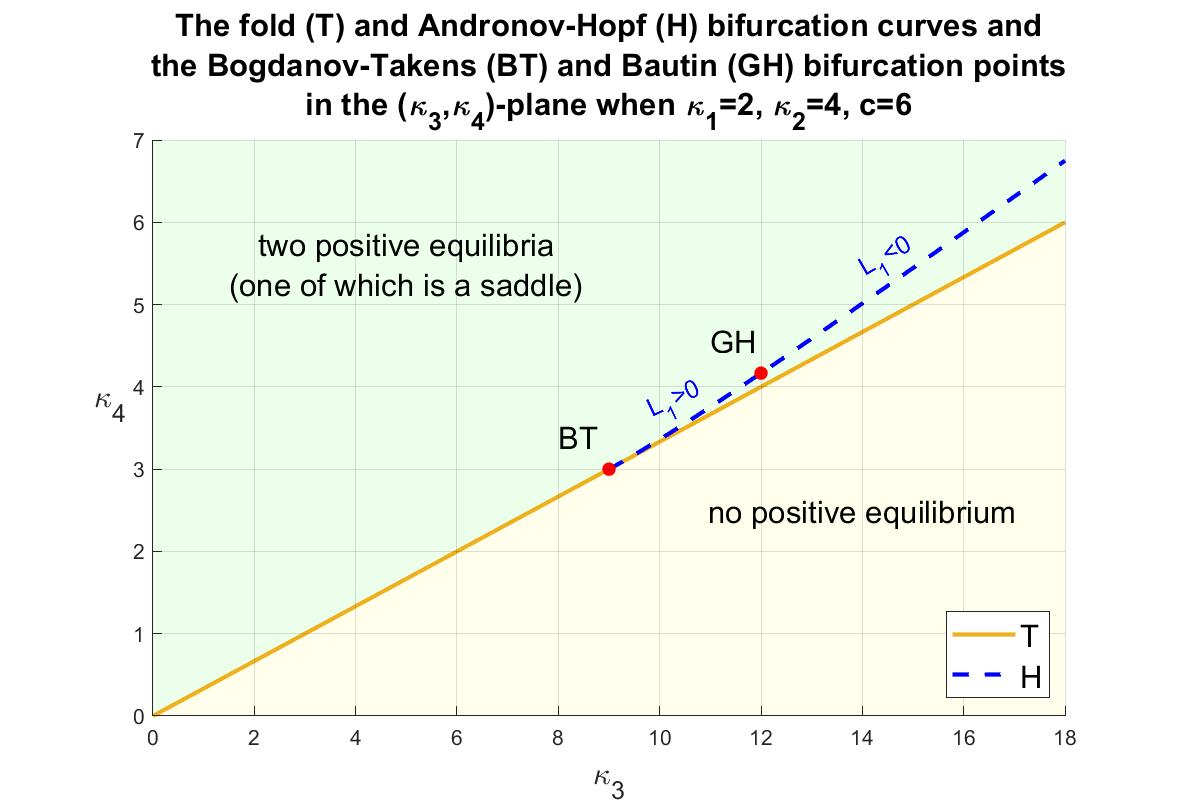}
\caption{Bifurcation diagram in the $(\kappa_3,\kappa_4)$-plane of the mass-action system (\ref{eq:ode_lifted_brusselator_2d}) with $\kappa_1$, $\kappa_2$ and $c$ fixed at $\kappa_1 = 2$, $\kappa_2 = 4$, $c=6$. On the solid ochre curve (T) a fold bifurcation of equilibria occurs. Along the dashed blue curve (H) there is an Andronov-Hopf bifurcation. A Bogdanov-Takens bifurcation occurs at (BT), where (H) touches (T). On the curve (H), the bifurcation is subcritical close to the (BT) point, but becomes supercritical beyond the point (GH), where the first focal value $L_1$ vanishes. At (GH), there is a Bautin bifurcation.}
\label{fig:bifurcation_diagram}
\end{figure}

\section*{References}

\bibliographystyle{abbrv}

\begin{thebibliography}{10}

\bibitem{abphopf}
D.~Angeli, M.~Banaji, and C.~Pantea.
\newblock Combinatorial approaches to {H}opf bifurcations in systems of
  interacting elements.
\newblock {\em Commun. Math. Sci.}, 12:1101--1133, 2014.

\bibitem{angelileenheersontag}
D.~Angeli, P.~{De Leenheer}, and E.~D. Sontag.
\newblock Graph-theoretic characterizations of monotonicity of chemical
  reaction networks in reaction coordinates.
\newblock {\em J. Math. Biol.}, 61(4):581--616, 2010.

\bibitem{banajiCRNosci}
M.~Banaji.
\newblock {Inheritance of oscillation in chemical reaction networks}.
\newblock {\em Appl. Math. Comput.}, 325:191--209, 2018.

\bibitem{banajiCRNosci1}
M.~Banaji.
\newblock Building oscillatory chemical reaction networks by adding reversible
  reactions.
\newblock {\em SIAM J. Appl. Math.}, 80(4):1751--1777, 2020.

\bibitem{banajisplitreacs}
M.~Banaji.
\newblock Splitting reactions preserves nondegenerate behaviours in chemical
  reaction networks, 2022.
\newblock \url{https://arxiv.org/abs/2201.13105.pdf}.

\bibitem{banajicraciun2}
M.~Banaji and G.~Craciun.
\newblock Graph-theoretic approaches to injectivity and multiple equilibria in
  systems of interacting elements.
\newblock {\em Commun. Math. Sci.}, 7(4):867--900, 2009.

\bibitem{banajipantea}
M.~Banaji and C.~Pantea.
\newblock Some results on injectivity and multistationarity in chemical
  reaction networks.
\newblock {\em SIAM J. Appl. Dyn. Syst.}, 15(2):807--869, 2016.

\bibitem{banajipanteaMPNE}
M.~Banaji and C.~Pantea.
\newblock The inheritance of nondegenerate multistationarity in chemical
  reaction networks.
\newblock {\em SIAM J. Appl. Math.}, 78(2):1105--1130, 2018.

\bibitem{balazsgithub}
B.~Boros.
\newblock {Reaction Networks GitHub repository}, 2022.
\newblock \url{https://github.com/balazsboros/reaction_networks}.

\bibitem{boros:hofbauer:2022a}
B.~Boros and J.~Hofbauer.
\newblock Limit cycles in mass-conserving deficiency-one mass-action systems,
  2022.
\newblock \url{https://arxiv.org/abs/2202.10406.pdf}.

\bibitem{boros:hofbauer:2022b}
B.~Boros and J.~Hofbauer.
\newblock Some minimal bimolecular mass-action systems with limit cycles, 2022.
\newblock \url{https://arxiv.org/abs/2202.11034.pdf}.

\bibitem{cheong}
R.~Cheong and A.~Levchenko.
\newblock Oscillatory signaling processes: the how, the why and the where.
\newblock {\em Curr. Opin. Genet. Dev.}, 20(6):665--669, 2010.

\bibitem{craciun}
G.~Craciun and M.~Feinberg.
\newblock Multiple equilibria in complex chemical reaction networks: {I}. {T}he
  injectivity property.
\newblock {\em SIAM J. Appl. Math.}, 65(5):1526--1546, 2005.

\bibitem{erdi:toth:1989}
P.~\'Erdi and J.~T\'oth.
\newblock {\em Mathematical models of chemical reactions. Theory and
  applications of deterministic and stochastic models}.
\newblock Princeton University Press, Princeton, 1989.

\bibitem{farkas:noszticzius:1985}
H.~Farkas and Z.~Noszticzius.
\newblock Generalized {L}otka-{V}olterra schemes and the construction of
  two-dimensional explodator cores and their {L}iapunov functions via
  ‘critical’ {H}opf bifurcations.
\newblock {\em J. Chem. Soc.{,} Faraday Trans. 2}, 81(10):1487--1505, 1985.

\bibitem{feinberg}
M.~Feinberg.
\newblock Chemical reaction network structure and the stability of complex
  isothermal reactors - {I}. {T}he deficiency zero and deficiency one theorems.
\newblock {\em Chem. Eng. Sci.}, 42(10):2229--2268, 1987.

\bibitem{feliuwiufAMC2012}
E.~Feliu and C.~Wiuf.
\newblock Preclusion of switch behavior in networks with mass-action kinetics.
\newblock {\em Appl. Math. Comput.}, 219:1449--1467, 2012.

\bibitem{feliuwiufInterface2013}
E.~Feliu and C.~Wiuf.
\newblock Simplifying biochemical models with intermediate species.
\newblock {\em J. Roy. Soc. Interface}, 10:20130484, 2013.

\bibitem{Fenichel71}
N.~Fenichel.
\newblock Persistence and smoothness of invariant manifolds for flows.
\newblock {\em Indiana Univ. Math. J.}, 21(3):193--226, 1971.

\bibitem{Fenichel79}
N.~Fenichel.
\newblock Geometric singular perturbation theory for ordinary differential
  equations.
\newblock {\em J. Differ. Equ.}, 31(1):53--98, 1979.

\bibitem{HPS77}
M.~W. Hirsch, C.~C. Pugh, and M.~Shub.
\newblock {\em {Invariant Manifolds}}, volume 583 of {\em Lecture Notes in
  Mathematics}.
\newblock Springer-Verlag, 1977.

\bibitem{hornjackson}
F.~Horn and R.~Jackson.
\newblock General mass action kinetics.
\newblock {\em Arch. Ration. Mech. Anal.}, 47(2):81--116, 1972.

\bibitem{joshishiu}
B.~Joshi and A.~Shiu.
\newblock Atoms of multistationarity in chemical reaction networks.
\newblock {\em J. Math. Chem.}, 51(1):153--178, 2013.

\bibitem{Kholodenko.2000aa}
B.~N. Kholodenko.
\newblock Negative feedback and ultrasensitivity can bring about oscillations
  in the mitogen-activated protein kinase cascades.
\newblock {\em Eur. J. Biochem.}, 267(6):1583--1588, 2000.

\bibitem{kuznetsov:2004}
Y.~A. Kuznetsov.
\newblock {\em Elements of Applied Bifurcation Theory}, volume 112 of {\em
  Applied Mathematical Sciences}.
\newblock Springer-Verlag, New York, third edition, 2004.

\bibitem{li_1996}
M.~Y. Li.
\newblock {Dulac criteria for autonomous systems having an invariant affine
  manifold}.
\newblock {\em J. Math. Anal. Appl.}, 199(2):374--390, 1996.

\bibitem{lin2021multistationarity}
K.~Lin, X.~Tang, and Z.~Zhang.
\newblock Multistationarity of reaction networks with one-dimensional
  stoichiometric subspaces, 2021.
\newblock \url{https://arxiv.org/abs/2108.09695}.

\bibitem{markevich}
N.~I. Markevich, J.~B. Hoek, and B.~N. Kholodenko.
\newblock {Signaling switches and bistability arising from multisite
  phosphorylation in protein kinase cascades}.
\newblock {\em J. Cell Biol.}, 164(3):353--359, 2004.

\bibitem{novaktyson}
B.~Nov\'ak and J.~J. Tyson.
\newblock Design principles of biochemical oscillators.
\newblock {\em Nature Rev. Mol. Cell Biol.}, 9:981--991, 2008.

\bibitem{obatake}
N.~Obatake, A.~Shiu, X.~Tang, and A.~Torres.
\newblock {Oscillations and bistability in a model of ERK regulation}.
\newblock {\em J. Math. Biol.}, 79:1515--1549, 2019.

\bibitem{Ortega.2006ab}
F.~Ortega, J.~L. Garc{\'e}s, F.~Mas, B.~N. Kholodenko, and M.~Cascante.
\newblock Bistability from double phosphorylation in signal transduction.
\newblock {\em FEBS Journal}, 273(17):3915--3926, 2006.

\bibitem{panteapersistence}
C.~Pantea.
\newblock On the persistence and global stability of mass-action systems.
\newblock {\em SIAM J. Math. Anal.}, 44(3):1636--1673, 2012.

\bibitem{Qiao.2007aa}
L.~Qiao, R.~B. Nachbar, I.~G. Kevrekidis, and S.~Y. Shvartsman.
\newblock {Bistability and oscillations in the Huang-Ferrell model of MAPK
  signaling}.
\newblock {\em PLoS Comput. Biol.}, pages 1819--1826, 2007.

\bibitem{schlogl:1971}
F.~Schl{\"o}gl.
\newblock On thermodynamics near a steady state.
\newblock {\em Z. Physik}, 248(5):446--458, 1971.

\bibitem{schlogl:1972}
F.~Schl{\"o}gl.
\newblock Chemical reaction models for non-equilibrium phase transitions.
\newblock {\em Z. Physik}, 253(2):147--161, 1972.

\bibitem{schneider:1969}
K.~R. Schneider.
\newblock \"{U}ber die periodischen {L}\"{o}sungen einer {K}lasse nichtlinearer
  autonomer {D}ifferentialgleichungssysteme dritter {O}rdnung.
\newblock {\em Z. Angew. Math. Mech.}, 49(7):441--443, 1969.

\bibitem{shinarfeinbergconcord1}
G.~Shinar and M.~Feinberg.
\newblock Concordant chemical reaction networks.
\newblock {\em Math. Biosci.}, 240(2):92--113, 2012.

\bibitem{simon:1992}
P.~L. Simon.
\newblock The reversible {LVA} model.
\newblock {\em J. Math. Chem.}, 9(4):307--322, 1992.

\bibitem{HalSmithJMC}
H.~Smith.
\newblock {Global dynamics of the smallest chemical reaction system with Hopf
  bifurcation}.
\newblock {\em J. Math. Chem.}, 50(4):989--995, 2012.

\bibitem{switchclock}
J.~J. Tyson, R.~Albert, A.~Goldbeter, P.~Ruoff, and J.~Sible.
\newblock Biological switches and clocks.
\newblock {\em J. R. Soc. Interface}, 5:S1--S8, 2008.

\bibitem{Wiggins}
S.~Wiggins.
\newblock {\em Introduction to Applied Nonlinear Dynamics and Chaos}.
\newblock Springer, 2003.

\bibitem{ye:1986}
Y.-Q. Ye~et al.
\newblock {\em Theory of Limit Cycles}.
\newblock Translations of Mathematical Monographs, Vol. 66, American
  Mathematical Society, Providence, R. I., 1986.

\end{thebibliography}

\end{document}